\documentclass[12pt,letterpaper,reqno]{amsart}
\usepackage{graphicx}
\usepackage{amsmath}
\usepackage{amssymb}
\usepackage{amsfonts}
\usepackage{amsrefs}
\usepackage{enumerate}
\usepackage[mathscr]{eucal}

\usepackage[usenames]{color}

\newcommand{\R}{\mathbb{R}}

\newtheorem{thm}{Theorem}
\newtheorem{lemma}[thm]{Lemma}
\newtheorem{prop}[thm]{Proposition}

\theoremstyle{definition}
\newtheorem{definition}[thm]{Definition}

\theoremstyle{remark}
\newtheorem{remark}{Remark}

\usepackage[%
bookmarks=true,
colorlinks=true,
linkcolor=black,
urlcolor=black,
citecolor=black,
plainpages=false,
pdfpagelabels,
final]{hyperref}
\usepackage[capitalise,nameinlink]{cleveref}

\def\madm{m_{\mathrm{ADM}}}
\def\miso{m_{\mathrm{iso}}}
\def\misoc{m_{\mathrm{iso}}^{\mathrm{c}}}
\def\misoql{m_{\mathrm{iso}}^{\mathrm{QL}}}
\def\tmiso{m_{\mathrm{iso}}^{\mathrm{H}}}

\title{A note on Huisken's isoperimetric mass}
\date{\today}
\author{Jeffrey L. Jauregui}
\address{Union College, Department of Mathematics, 807 Union St.,
Schenectady, NY 12308, USA}
\email{jaureguj@union.edu}
\author{Dan A. Lee}
\address{Graduate Center and Queens College, City University of New York, 365 Fifth Avenue,
New York, NY 10016, USA}
\email{dan.lee@qc.cuny.edu}
\author{Ryan Unger}
\address{Stanford University, Department of Mathematics,  Building 380, Stanford, CA 94305, USA \&
 UC Berkeley, Department of Mathematics, Evans Hall, Berkeley, CA 94720, USA}
\email{runger@stanford.edu}

\begin{document}

\begin{abstract}
In this short note, we explain that Huisken's isoperimetric mass is always nonnegative for elementary reasons.
\end{abstract}

\maketitle 
The \emph{ADM mass}, $\madm$, is well known to represent the ``total mass'' of an asymptotically flat manifold \cite{ADM}. Although the formula for the ADM mass involves derivatives of the metric, one may desire a notion of mass for metrics that are not regular enough for the ADM mass to be defined. To this end, G.~Huisken proposed the \emph{isoperimetric mass} as a stand-in for the ADM mass for low regularity metrics~\cites{Hui1,Hui2, Hui3}. We now recall the precise definitions.

\begin{definition}\label{def_AF}
Let $n\ge3$, and let $M$ be a smooth $n$-manifold (possibly with boundary) equipped with a $C^0$ Riemannian metric $g$.
We say that $(M, g)$ is a \emph{one-ended\footnote{Note that according to our definition, $(M, g)$ is necessarily complete. Modifications to our results for manifolds with multiple ends are straightforward.} $C^0$-asymptotically flat manifold} if there exists a compact set $K\subset M$ such that $M\smallsetminus K$ is diffeomorphic to $\R^n \smallsetminus \bar{B}$  (for some closed coordinate ball $\bar{B}$) and such that in this coordinate chart (which we refer to as an \emph{asymptotically flat end}),
\[ |g_{ij} - \delta_{ij}| \leq C |x|^{-\tau} \]
for some constants $C,\tau > 0$.
\end{definition}
This definition should be contrasted with the more common (stronger) concept of \emph{one-ended $C^2$-asymptotically flat manifolds}, which also requires decay of the first two derivatives of $g_{ij}$, and typically also requires $\tau > \frac{n-2}{2}$ and integrability of the scalar curvature. See, for example, Definition 3.5 of \cite{Lee:book}. These manifolds have a well-defined ADM mass. Huisken made the remarkable observation that the ADM mass can be detected by isoperimetric quantities, leading to the following definitions.

\begin{definition}[Huisken] \label{def_miso}
Let $(M, g)$ be a one-ended $C^0$-asymptotically flat $3$-manifold\footnote{These definitions generalize to higher dimensions in a straightforward way (see the appendix of \cite{Jau}), but we will keep our presentation to three dimensions for simplicity.}.  Let $\Omega$ be a bounded open subset of $M$ with finite perimeter. Then the \emph{quasilocal isoperimetric mass} of $\Omega$ is
\[ \misoql(\Omega) := \frac{2}{|\partial^* \Omega|} \left(|\Omega| - \frac{1}{6\sqrt{\pi}} |\partial^* \Omega|^{3/2} \right), \]
where  $|\Omega|$ and $|\partial^* \Omega|$ denote the volume and perimeter\footnote{See the appendix of \cite{JL2} for a discussion of perimeter for $C^0$ Riemannian metrics. In general, if $\partial \Omega$ is a regular surface, then $|\partial^* \Omega| = |\partial \Omega|$, the area of $\partial \Omega$.} of $\Omega$, respectively.
\begin{enumerate}
\item In \cite{Hui1}, Huisken defined 
\[ \misoc(M,g) :=  \limsup_{R \to \infty} \misoql(\mathcal{B}_R),\]
where $\mathcal{B}_R$ is the region of $M$ enclosed by the coordinate sphere of radius $R$.
\item In \cite{Hui3}, Huisken defined
\[ \tmiso(M,g) :=  \limsup_{|\partial^*\Omega| \to \infty} \misoql(\Omega) ,\]
where the limit superior is taken over all bounded open subsets of $M$ with finite perimeter.
\item It is also convenient to define (as in~\cites{EM, JL}, for example)
\[ \miso(M, g) := \sup_{\{\Omega_i\}} \left(\limsup_{i \to \infty} \misoql(\Omega_i)\right),\]
where the supremum is taken over all sequences $\{\Omega_i\}$ of bounded open sets with finite perimeter that exhaust $M$.
\end{enumerate}
\end{definition}
From the definitions it is clear that
\[ \misoc \le \miso \le \tmiso.\]
We decorate $\misoc$ to emphasize that it is a ``coordinate'' based definition, and $\tmiso$ to indicate that it is Huisken's original version of the definition, but we choose to regard $\miso$ as the main definition of \emph{isoperimetric mass}, because it is easy to see that $\misoc$ and $\miso$ are purely asymptotic quantities which depend only on the geometry in a neighborhood of infinity. We define the isoperimetric mass of an asymptotically flat \emph{end} by arbitrarily filling it in and then applying Definition \ref{def_miso}:

\begin{definition}
Let $(\mathcal E, g)$ be an asymptotically flat end. We define $\misoc(\mathcal E, g)$ and $\miso(\mathcal E, g)$ to be the corresponding $\misoc$ and $\miso$ of  \emph{any} one-ended asymptotically flat manifold in which $(\mathcal E,g)$ is isometrically embedded.
\end{definition}

Huisken asserted that for any one-ended $C^2$-asymptotically flat $3$-manifold with nonnegative scalar curvature (whose boundary is either empty or minimal), $\tmiso$ should be equal to $\madm$. This was established by the following sequence of results. First, we have the following theorem, proved in \cite{FST}, where P.~Miao is credited for the argument.
\begin{thm}[Miao]\label{thm_FST}
Let $(\mathcal E,  g)$ be $C^2$-asymptotically flat end of any dimension. Then $\misoc=\madm$.
\end{thm}
We emphasize that this result is an asymptotic computation that does not involve scalar curvature. The next result is a simple computation due to the first- and second- named authors. 
\begin{prop}[Proposition 37 of \cite{JL}]
Let $(M, g)$ be a one-ended $C^0$-asymptotically flat manifold of any dimension. Then $\tmiso=\miso$.
\end{prop}
In other words, $\tmiso$ is the same as $\miso$ on the spaces contemplated by Definition~\ref{def_miso}. However, note that the definition of $\tmiso$ immediately generalizes to any Riemannian manifold, in which case it will depend on the geometry of all of the ends, and hence may not equal the value of $\miso$ on a given asymptotically flat end.\footnote{There is yet another variant of isoperimetric mass appearing in Section 7 of~\cite{EM}, which is defined in terms of the isoperimetric profile function of $(M, g)$. This variant is also the same as $\miso$ and $\tmiso$ for one-ended $C^0$-asymptotically flat manifolds but possibly different on more general spaces.}

\begin{remark}
  Define negative mass Schwarzschild space to be
  \begin{equation*}
      M=\R^3\smallsetminus\left\{|x|\le \frac{|m|}{2}\right\},\quad g_{ij}=\left(1+\frac{m}{2|x|}\right)^4\delta_{ij},
  \end{equation*}
  where $m<0$. This is an incomplete Riemannian manifold with one asymptotically flat end and one end corresponding to the singularity, and  its $\tmiso$ equals the $\miso$ of its asymptotically flat end. This follows from the argument of \cite[Proposition 37]{JL} due to the existence of an isoperimetric inequality for ``sufficiently large'' sets: there exist constants $C,c>0$ such that if  $\Omega\subset M$ is a set of finite perimeter for which $|\partial^*\Omega|\ge c$, then $|\Omega|\le C|\partial^*\Omega|^{3/2}$.  
\end{remark}

Finally, the first- and second- named authors fully justified Huisken's assertion in~\cite{JL}. Another proof was later given by O.~Chodosh, M.~Eichmair, Y.~Shi, and H.~Yu in \cite{CESY}.
\begin{thm}[\cites{JL, CESY}]\label{thm:JL}
Let $(M,g)$ be a one-ended $C^2$-asymptotically flat $3$-manifold with nonnegative scalar curvature, whose boundary is either empty or minimal. Then $\miso=\madm$.
\end{thm}

Unlike the previous two results, Theorem~\ref{thm:JL} is a global theorem that relies on completeness and nonnegative scalar curvature in an essential way. In fact, it fails without those global hypotheses. Recent work of G.~Antonelli, M.~Fogagnolo, S.~Nardulli, and M.~Pozzetta~\cite{AFNP} shows that if $(\mathcal E,g)$ is a $3$-dimensional $C^0$-asymptotically flat \emph{end} with nonnegative scalar curvature (in a suitable weak sense), then $\miso\ge0$, even if $\madm<0$.
 In particular, any \emph{negative mass} Schwarzschild space actually has \emph{nonnegative} isoperimetric mass. The proof of \cite{AFNP} uses the notion of inverse mean curvature flow from a point (\`a la Huisken--Ilmanen \cite{HI}), results of Y.~Shi on estimating the volumes and areas of level sets of inverse mean curvature flow in the presence of nonnegative scalar curvature \cite{Shi}, and some technical results about $C^0$ Riemannian metrics. 

\medskip

The main purpose of this note is to point out that $\miso$ is always nonnegative, and that this fact has nothing to do with scalar curvature. It is a direct consequence of the elementary fact that large and far off-center coordinate balls in $C^0$-asymptotically flat ends have volumes and perimeters close to those of  Euclidean balls.

\begin{thm}
\label{thm_main}
Let $(\mathcal E, g)$ be a $C^0$-asymptotically flat end of any dimension. Then $\miso\ge0$.
\end{thm}

In retrospect, the fact that negative mass Schwarzschild actually has $\miso > \madm$ is not surprising: Huisken's definition of isoperimetric mass is based on the relationship between mass, area, and enclosed volume for the symmetry spheres in Schwarzschild spaces, and the fact that those symmetry spheres are isoperimetric surfaces when the mass of the Schwarzschild metric is positive \cite{Bray}. However, it is known that the symmetry spheres in negative mass Schwarzschild spaces are \emph{not} isoperimetric (as they are unstable CMC surfaces) \cite{CGGK}. It is natural to wonder what the value of
$\miso$ is for negative mass Schwarzschild, as it is related to understanding the isoperimetric properties of these spaces.

One of the main motivations for studying the isoperimetric mass came from a desire to generalize the celebrated positive mass theorem \cites{SY, Wit} to a low-regularity setting. That is, given a one-ended $C^0$-asymptotically flat manifold with nonnegative scalar curvature (in some weak sense), can one prove that (some weakened version of) mass is nonnegative? Huisken \cite{Hui2} expressed the hope that this\footnote{Huisken's preferred notion of weak scalar curvature in this context is defined in relation to isoperimetric properties of small balls. In particular, it differs from the one used in~\cite{AFNP}.} might be true for $\miso$, but our observation clearly renders this conclusion to be vacuously true. One might still wonder whether a \emph{rigidity} statement might hold: that is, if such a manifold has $\miso=0$, must it be Euclidean space? Another question is whether a low-regularity generalization of the positive mass theorem might hold for $\misoc$.

The proof of Theorem \ref{thm_main} is elementary, and we break it down into two lemmas. For simplicity, we write the proof in three dimensions, but the modifications needed for higher dimensions are trivial. Note that all volumes and perimeters below are computed with respect to $g$. 
\begin{lemma}
\label{lemma1}
Let $(M,g)$ be a one-ended $C^0$-asymptotically flat $3$-manifold, and let $\Omega$ be a bounded subset of $M$. For any $V>0$ and $\varepsilon>0$, there exists a bounded region $E\subset M$ with smooth boundary such that $E$ has positive distance from $\Omega$ and satisfies
\begin{align*}
|E| &\geq V,  \\
|\partial E| &\geq (36\pi)^{1/3} V^{2/3},\\
\left||E| - \frac{1}{6\sqrt{\pi}} |\partial E|^{3/2}\right| &\le \varepsilon.
\end{align*}    
\end{lemma}

\begin{proof}
Let $V>0$ and choose $R$ large enough so that  $\tfrac{4\pi}{3} R^3  > V$.
For sufficiently large $i$, let $E_i$ be the coordinate ball of radius $R$, centered at $(i,0,0)$ in the asymptotically flat coordinate chart for $M$. For all sufficiently large $i$, $E_i$ has positive distance from $\Omega$, and
 by $C^0$-asymptotic flatness,
\begin{align*}
    \lim_{i \to \infty} |E_i| &= \tfrac{4\pi}{3} R^3  > V ,\\
    \lim_{i \to \infty} |\partial E_i| &= 4\pi R^2 > (36\pi)^{1/3} V^{2/3}.
\end{align*}
In particular,
\[ \lim_{i \to \infty} \left(|E_i| - \frac{1}{6\sqrt{\pi}} |\partial E_i|^{3/2}\right)=0.\]
The claim follows by taking $E$ to be $E_i$ for $i$ sufficiently large.
\end{proof} 

\begin{lemma}
\label{lemma2}
Let $(M,g)$ be a one-ended $C^0$-asymptotically flat $3$-manifold, and let $\Omega$ be a bounded subset of $M$.
For any $\varepsilon>0$, there exists a bounded open set $\tilde \Omega \supset \Omega$ with smooth boundary such that
\begin{align*}
|\misoql(\tilde\Omega)|\le\varepsilon.
\end{align*}    
\end{lemma}

\begin{proof}
Without loss of generality, if necessary, replace $\Omega$ with a bounded, open superset that has smooth boundary. Applying Lemma \ref{lemma1}, we can construct a sequence of bounded open sets $E_i\subset M$ with smooth boundary such that $E_i$ has positive distance from $\Omega$, $|E_i| \nearrow \infty$, $|\partial E_i| \nearrow \infty$, and
\[|E_i| - \frac{1}{6\sqrt{\pi}} |\partial E_i|^{3/2} \to 0\] as $i\to\infty$. 
We define $\tilde\Omega_i=E_i \cup \Omega$. It then holds that
\begin{align*}
    |\tilde \Omega_i | &= |E_i|+ |\Omega |  ,\\
    |\partial \tilde\Omega_i| &=  |\partial E_i|+ |\partial \Omega| .
\end{align*}

We compute
\begin{align*}
|\misoql(\tilde\Omega_i)|&=\frac{2}{|\partial \tilde \Omega_i|}\left||\tilde \Omega_i| - \frac{1}{6\sqrt{\pi}} |\partial \tilde \Omega_i|^{3/2}\right|\\ &= \frac{2}{|\partial  E_i| +|\partial \Omega|}\left| |E_i|+|\Omega|  - \frac{1}{6\sqrt{\pi}} \left( |\partial E_i| +|\partial \Omega|\right)^{3/2}\right|\\
&= \frac{1}{1+\frac{|\partial \Omega|}{|\partial E_i|}} \frac{2}{|\partial  E_i| }\left||E_i| + |\Omega| - \frac{1}{6\sqrt{\pi}} |\partial E_i|^{3/2} + O\left(|\partial E_i|^{1/2}\right)\right|\\
&\leq \frac{1}{1+\frac{|\partial \Omega|}{|\partial E_i|}} \frac{2}{|\partial E_i| }\left||E_i|  - \frac{1}{6\sqrt{\pi}} |\partial E_i|^{3/2}\right| + O\left(|\partial E_i|^{-1/2}\right),
\end{align*}
which limits to $0$ as $i \to \infty$. We then take $\tilde \Omega$ to be $\tilde \Omega_i$ for $i$ sufficiently large.
\end{proof}

\begin{proof}[Proof of Theorem \ref{thm_main}]
Let $(M,g)$ be any one-ended $C^0$-asymptotically flat $3$-manifold containing $(\mathcal E, g)$. Let $R_0>0$ be large enough so that the end $\mathcal E$ contains the coordinate sphere of radius $R_0$, and for all $R\ge R_0$, and let 
$\mathcal{B}_R$ denote the region of $M$ enclosed by the coordinate sphere of radius $R$. We construct a exhaustion sequence $\{\tilde\Omega_i\}$ by bounded open subsets of $M$ as follows: We start with $\tilde\Omega_1 = \mathcal{B}_{R_0}$, and
we inductively define $\tilde\Omega_{i+1}$ to be the resulting $\tilde\Omega$ obtained by applying Lemma~\ref{lemma2} to $\Omega=\tilde\Omega_{i}\cup \mathcal B_{R_0+i}$ with $\varepsilon=\frac{1}{i}$. This is clearly an exhaustion sequence with the property that 
$\displaystyle \lim_{i\to\infty} \misoql(\tilde\Omega_i) =0$, which implies that $\miso(M, g) \geq 0$.
\end{proof}

\medskip

\paragraph{\emph{Acknowledgments}:} R.U.~acknowledges support from the NSF grant DMS-2401966. We also thank G.~Antonelli for sharing an early preprint of~\cite{AFNP}.

\end{document}